\documentclass[10pt]{article}
\usepackage{amssymb}
\usepackage{latexsym }
\usepackage{graphicx }
\usepackage{amsfonts}
\newtheorem{lemma}{Lemma}

\newtheorem{corollary}[lemma]{Corollary}
\newtheorem{proposition}[lemma]{Proposition}

\catcode`\@=11
\catcode`\@=12
\newenvironment{proof}[1][Proof]{\textbf{#1.} }{\
\rule{0.5em}{0.5em}}

\def\Expect{\mathbb{E}}
\def\Prob{\mathbb{P}}

\def\Var{\mbox{Var}}
\def\Cov{\mbox{Cov}}

\usepackage{latexsym}
\usepackage[T1]{fontenc}
\usepackage{amsmath}
\usepackage{amssymb}
\usepackage{graphicx,graphpap}
\usepackage{color}
\usepackage{dsfont}

\newtheorem{ass}{Assumption}
\newtheorem{conj}{Conjecture}
\newcommand{\vc}[1]{\boldsymbol{#1}}

\begin{document}

\title{
The Second Order Terms of the Variance Curves\\
for Some Queueing Output Processes
}
\author{
\normalsize Sophie Hautphenne\thanks{
The University of Melbourne, Melbourne, Australia.},
\,
\normalsize Yoav Kerner\thanks{
Ben Gurion University, Beer Sheva, Israel.},
\, 
\normalsize Yoni Nazarathy\thanks{
The University of Queensland, Brisbane, Australia.
}
\, and
\normalsize Peter Taylor\thanks{
The University of Melbourne, Melbourne, Australia.
}.
}
\maketitle
\maketitle \thispagestyle{empty}

\begin{abstract}
We consider queueing output processes of some elementary queueing models such as the M/M/1/K queue and the M/G/1 queue. An important performance measure for these counting processes is their variance curve, indicating the variance of the number of served customers over a time interval.  Recent work has revealed some non-trivial properties dealing with the asymptotic rate at which the variance curve grows.  In this paper we add to the results by finding explicit expressions for the second order approximation of the variance curve, namely the y-intercept of the linear asymptote.

For M/M/1/K queues our results are based on the Drazin inverse of the generator.  It turns out that by viewing output processes as MAPs (Markovian Arrival Processes) and considering the Drazin inverse, one can obtain explicit expressions for the y-intercept, together with some further insight regarding the BRAVO effect (Balancing Reduces Asymptotic Variance of Outputs). For M/G/1 queues our results are based on a classic transform of D.J. Daley.  In this case we represent the y-intercept of the variance curve in terms of the first three moments of the service time distribution.

A further performance measure that we are able to calculate for both models, is the asymptotic covariance between the queue length and the number of arrivals or departures. In addition we shed light on a classic conjecture of Daley, dealing with characterization of stationary M/M/1 queues within the class of stationary M/G/1 queues, based on the variance curve.
\end{abstract}


\section{Introduction}
Many models in applied probability and stochastic operations research involve counting processes.  Such processes occur in supply chains, health care systems, communication networks as well as many other contexts involving service, logistics and/or technology. The canonical counting process example is the Poisson process. Generalizations include renewal processes, Markovian Arrival Processes (MAPs) (see for example \cite{bookLatoucheRamaswami1999}), or general simple point processes on the line (see for example \cite{daley2003itp}).  Counting process models arise naturally in applications involving the random occurrence of events.

Sometimes counting processes are applied in their own right, while at other times they constitute components of more complicated models such as queues, population processes or risk models.  In other instances, counting processes are implicitly defined and constructed through applied probability models. For example, a realization of a queue induces additional counting processes such as the departure process, $\{D(t),t \ge 0\}$, counting the number of serviced customers in the queue until time $t$.

Departure counting processes of queues have been heavily studied in classical applied probability, see for example \cite{Daley0510} and \cite{DisneyKonig0283} for extensive accounts. Nevertheless, many open questions remain, some of which have received attention in recent years. For example there are several open questions about the ability to describe $\{D(t)\}$ as a MAP as in \cite{bean2000map}, \cite{BeanEtAl0244}  and \cite{olivier1994existence}. Further, the discovery of the BRAVO effect (Balancing Reduces Asymptotic Variance of Outputs) has motivated research on the variability of departure processes of queues, particularly in critically loaded regimes. Recent papers on this topic are \cite{al2011asymptotic}, \cite{daley2011revisiting},\cite{daleyLeeNaz},   \cite{nazarathy2011variance} and \cite{NazarathyWeiss0336}.

Next to the mean curve, $m(t) := \Expect[D(t)]$, an almost equally important performance measure of a counting processes is the variance curve, $v(t):=\Var\big(D(t)\big)$. For example, for a Poisson process with rate $\alpha$, the variance curve
\[
v(t) = \alpha t
\] is the same as the mean curve.
For more complicated counting processes, the variance curve is not as simple and generally does not equal the mean curve.  For example, for a stationary (also known as {\em equilibrium}) renewal-process with inter-renewal times distributed as the sum of two independent exponential random variables, each with mean $(2 \alpha)^{-1}$, we have
\[
m(t) = \alpha t  - \frac{1}{4} + \frac{1}{4}e^{-4 \alpha t},
\qquad
v(t) = \alpha \frac{1}{2} t+ \frac{1}{8} -\frac{1}{8}e^{-4 \alpha t}.
\]
For the {\em ordinary} case of the same renewal process (the first inter-renewal time is distributed as all the rest) the variance curve is
\[
v(t) = \alpha\frac{1}{2} t+ \frac{1}{16} - te^{-4 \alpha t} - \frac{1}{16} e^{-8 \alpha t}.
\]
These explicit examples are taken from \cite{cox1962renewal}, Section 4.5.  In fact, for general, non-lattice, renewal processes (both equilibrium and ordinary), with inter-renewal times having a finite second moment, with squared coefficient of variation $c^2$, and mean $\alpha^{-1}$, it is well known that,
\begin{equation}
\label{eq:renewalAsympVar}
v(t) = \alpha c^2 t + o(t),
\end{equation}
yet in general, the full description of $v(t)$ (through the $o(t)$ term) is typically not as simple as in the examples above. Nevertheless, if the third moment of the inter-renewal time is finite, then
\begin{equation}
\label{eq:renewalAsymptote}
v(t) = 
\left\{
\begin{array}{ll}
\alpha c^2 t +\frac{5}{4}(c^4-1) - \frac{2}{3} (\gamma c^3 - 2) + o(1), & \mbox{for the equilibrium case,} \\
\alpha c^2 t + \frac{1}{2}(c^4-1) - \frac{1}{3}(\gamma c^3 -2) + o(1), & \mbox{for the ordinary case},
\end{array}
\right.
\end{equation}
where $\gamma$ is the skewness coefficient of the inter-renewal time\footnote{The skewness coefficient of a random variable $X$ is $\Expect\Big[\Big( \frac{X-\Expect[X]}{\sqrt{\Var(X)}} \Big)^3\Big]$.}. We remind the reader that for exponential random variables (making the renewal process a Poisson process), $c^2=1$ and $\gamma=2$, as indeed the ordinary and equilibrium versions a Poisson process are identical. See \cite{bookAsmussen2003} and \cite{daley2003itp} for more background on renewal processes. Equation \eqref{eq:renewalAsymptote} appears under a slightly different representation in \cite{cox1962renewal} and was essentially first found in \cite{smith1959cumulants}. Generalizations of renewal processes are in \cite{brown1975second}, \cite{daley1978asymptotic} and \cite{hunter1969moments}.

This emerging form of the variance curve,
\begin{equation}
\label{eq:varCurve}
v(t) = \overline{v} t + \overline{b} + o(1),
\end{equation}
is fruitful as it yields an asymptotically exact approximation for the variance curve for non-small $t$. A point to observe is that in general $\overline{b}$ depends on the version of the renewal process (ordinary vs. equilibrium) while $\overline{v}$ does not. We refer to $\overline{v}$ as the {\em  asymptotic variance rate} and to $\overline{b}$ as the {\em y-intercept}.  Since the latter depends on the initial conditions, we generally employ the notation $\overline{b}_e$ for the stationary (equilibrium) system, $\overline{b}_0$ for systems starting empty and $\overline{b}_\theta$ for systems with arbitrary initial conditions.

Moving on from renewal processes to implicitly defined counting processes, the variance curve is typically more complicated to describe and characterize. For example, while the output of a stationary M/M/1 queue with arrival rate $\lambda$ and service rate $\mu$,  is simply a Poisson process with rate $\lambda$  (see \cite{bookKelly1979}), the variance curve when the system starts empty at time $0$ is much more complicated than $v(t)=\lambda t$. It can be represented in terms of integrals of expressions involving Bessel functions of the first kind, and requires several lines to be written out fully (as in Theorem 5.1 of \cite{al2011asymptotic}). Nevertheless (see Theorem 5.2 in \cite{al2011asymptotic}) the curve is sensibly approximated as follows: 
\begin{equation}
\label{eq:mm1Var}
v(t) =
\left\{
\begin{array}{ll}
\lambda t - \frac{\rho}{(1-\rho)^2}  + o(1),      &          \mbox{{\rm if } $\rho < 1$,}
  \vspace{0.2cm}\\
  \vspace{0.2cm}
 2 (1-\frac{2}{\pi}) \lambda t -  \sqrt{\frac{\lambda}{\pi}}\,t^{1/2}+\frac{\pi-2}{4 \pi} + o(1),
  & \mbox{{\rm if } $\rho=1$,} \\
\mu     t - \frac{\rho}{(1-\rho)^2}  + o(1),      &          \mbox{{\rm if } $\rho > 1$,}
\end{array}
\right.
\end{equation}
where $\rho := \lambda/\mu$.

As observed from the formula above, it may be initially quite surprising that the asymptotic variance rate is reduced by a factor of $2(1-2/\pi) \approx 0.73$ when $\rho$ changes from being approximately $1$ to exactly~$1$. This is a manifestation of the BRAVO effect. BRAVO was first observed for M/M/1/K queues in \cite{NazarathyWeiss0336} in which case, as $K\rightarrow\infty$, the factor is $2/3$ as is (re)demonstrated in this paper. It was later analyzed for M/M/1 queues and more generally GI/G/1 queues in \cite{al2011asymptotic}. BRAVO was numerically conjectured for GI/G/1/K queues in \cite{nazarathy2011variance}, and observed for multi-server M/M/s/K queues in the many-server scaling regime in \cite{daleyLeeNaz}.

Our focus in this paper is on the more subtle y-intercept, $\overline{b}$ term. For a stationary M/M/1, $\{D(t)\}$ is a Poisson process and thus $\overline{b}_e=0$. As opposed to that, for the $M/M/1$ queue starting empty, it follows from \eqref{eq:mm1Var} that $\overline{b}_0 = -\rho/(1-\rho)^2$ as long as $\rho \neq 1$. When $\rho =1$, we also see from \eqref{eq:mm1Var} that the y-intercept does not exist as there is no linear asymptote for the variance curve. This can happen more generally: If for example there is sufficient long range dependence in the counting process, then the variance can grow super-linearly (see \cite{daley1997long} for some examples). This demonstrates that the asymptotic variance rate, $\overline{v}$, and the y-intercept, $\overline{b}$, need not exist for every counting process.  Nevertheless, for a variety of models and situations, both $\overline{v}$ and $\overline{b}$ exist, and thus the linear asymptote is well-defined. In such cases, having a closed formula is beneficial for performance analysis of the model at hand. 

We are now faced with the challenge of finding the y-intercept term for other counting processes generated by queues. In this paper we carry out such an analysis for two models related to the M/M/1 queue: a finite capacity M/M/1/K queue, and an infinite capacity M/G/1 queue. Besides obtaining explicit formulas for $\overline{b}_e, \overline{b}_0$ and $\overline{b}_\theta$, our investigation also pinpoints some of the analytical challenges involved and raises some open questions. Here is a summary of our main contributions:
\begin{description}
\item {\bf M/M/1/K queues:}  In this case the departure process is a MAP. The linear asymptote is then given by formulas based on the matrix $\Lambda^-:=({\mathbf 1} \vc\pi - \Lambda )^{-1}$, where $\Lambda$ is the generator matrix of the (finite) birth-death process, $\vc\pi$ is its stationary distribution taken as a row vector, and ${\mathbf 1}$ is a column vector of $1$'s. In the case where $\rho=1$, the distribution $\vc\pi$ is uniform and an explicit expression for $\Lambda^-$ was previously found, which in turn yielded the equilibrium version of the y-intercept in Proposition~4.4 of  \cite{NazarathyWeiss0336}:
\[
\overline{b}_e = 
\frac{7K^4+28K^3+37K^2+18K}{180(K+1)^2}.
\]
When $\rho \neq 1$, the form of the inverse $\Lambda^-$ is more complicated and an expression for $\overline{b}$ has not been previously known.  We are now able to find such an expression for both the stationary version and for arbitrary initial conditions. Our results are based on relating $\Lambda^-$ to the matrix
\[
\Lambda^{\sharp}=\int_0^\infty (P(t) - \vc 1\vc \pi) \, dt,
\]
where $P(\cdot)$ is the transition probability kernel of the birth-death process. The matrix $\Lambda^{\sharp}$ is called the \textit{Drazin inverse} of $\Lambda$, and we are able to provide explicit expressions for the entries of this matrix.  Our contribution also encompasses some useful results regarding arbitrary MAPs (Markovian Arrival Processes) which, to the best of our knowledge, have not appeared elsewhere. These results are used to find $\overline{b}_\theta$, as well as related covariances of the M/M/1/K queue.
\item {\bf Stable M/G/1 queues with finite third moment of G:} Having a finite third moment for the service time distribution ensures that queue lengths have a finite variance. When the service time distribution is not exponential, the form of $\overline{b}$ was previously not known. Our contribution is in finding an exact expression for the $\overline{b}$ term based on the first three moments of $G$. We begin with $\overline{b}_e$, after which we employ a simple coupling argument to find $\overline{b}_\theta$ and $\overline{b}_0$.
\end{description}

The structure of the rest of the paper is as follows: In Section~\ref{sec:MM1K} we present our M/M/1/K queue results for $\overline{b}$ together with a discussion of the Drazin inverse and its application for MAPs. Further, we find related covariances. In Section~\ref{sec:mainResMG1} we present our M/G/1 queue results for $\overline{b}$. In addition we discuss a related conjecture of Daley, dealing with a characterization of the M/M/1 queue within the class of stationary M/G/1 queues. We also find related covariances in the system.    We conclude in Section~\ref{sec:conc}.


\section{The M/M/1/K Queue}
\label{sec:MM1K}

We begin our investigation with the M/M/1/K queue, where $K$ denotes the total capacity of the system. In this case, it is well known that the departure process $\{D(t)\}$ is a MAP and is only a renewal processes when $K=1$. Some standard references on MAPs are  \cite{bookAsmussen2003} and \cite{bookLatoucheRamaswami1999}. To avoid confusion, we note that in \cite{bookAsmussen2003} these types of processes are referred to as MArPs.

Denote the arrival rate by $\lambda >0$, the service rate by $\mu>0$ and let $\rho := \lambda/\mu$ be the traffic intensity. The queue length process, $\{Q(t)\}$, is a continuous-time Markov chain on the state space $\{0,1,\ldots,K\}$, with generator matrix $\Lambda$ and stationary distribution (row) vector $\vc\pi$ given by
$$
\Lambda=\left[\begin{array}{ccccc} 
-\lambda & \lambda &  &  &0\\ 
\mu &-(\mu+\lambda) & \lambda &  & \\ 
& \ddots & \ddots & \ddots & \\
& & \mu &-(\mu+\lambda) & \lambda\\
0 & & &\mu &-\mu
\end{array}\right],
\qquad 
\vc\pi=
\left\{\begin{array}{ll}
 \dfrac{1-\rho}{1-\rho^{K+1}}\left[1,\rho,\rho^ 2,\ldots,\rho^ K\right],
 & \textrm{for }\rho\neq 1,\\[1em]
  \dfrac{1}{K+1}\,\boldsymbol{1}',&\textrm{for }\rho= 1.\end{array}\right. 
$$
The departure process $\{D(t)\}$ is a MAP of which the phase-process is $\{Q(t)\}$, and the event intensity matrix $\Lambda_1$ is given by
$$
\Lambda_1
=
\left[\begin{array}{ccccc} 0 & 0 &  &  &0\\
\mu &0 & 0&  & \\ 
& \ddots & \ddots & \ddots & \\
& & \mu &0 &0\\
0 & & &\mu&0\end{array}\right].
$$
In brief, $\Lambda_1$ indicates which transitions of $\{Q(t)\}$ will count as increments of $\{D(t)\}$.
Denote $\Lambda^-:=(\vc 1\vc \pi-\Lambda)^{-1}$ and refer to this matrix as the \textit{fundamental matrix}. The fundamental matrix is probably the best known and most widely used generalised inverse for the generator of a Markov chain. Another generalised inverse which has a clear probabilistic interpretation and allows us to obtain explicit expressions for $\bar{v}$ and $\bar{b}$, is given by the Drazin inverse,
$$
\Lambda^{\sharp}:=\int_0^{\infty} (e^{\Lambda t}-\vc1\vc\pi) \,dt.
$$ 
The Drazin inverse can be interpreted as a measure of the total deviation from the limiting probabilities, which is why it is also refered to as the \textit{deviation matrix} in the literature, see for instance \cite{CoolenSc}. The fundamental matrix $\Lambda^-$ and the Drazin inverse $\Lambda^{\sharp}$ are related in the following way:
\begin{equation}\label{rel}
\Lambda^-=\Lambda^{\sharp}+\vc1\vc\pi. 
\end{equation}
For finite state space continuous-time Markov chains, such as the M/M/1/K queue, the fundamental matrix and the Drazin inverse always exist. The Drazin inverse satisfies the  properties
\begin{eqnarray}
\label{p1}
\Lambda^{\sharp}\vc1&=&\vc 0,\\\label{p2} \vc\pi \Lambda^{\sharp}&=&\vc 0, \\ \nonumber \Lambda^{\sharp} \Lambda &=&\Lambda \Lambda^{\sharp}\;=\;\vc 1\vc\pi-I,
\end{eqnarray}
as well as
$$
\Lambda^{\sharp}_{i,j}=\pi_{j}\,(m^e_{j}-m_{i,j}),
$$ 
where $m_{i,j}$ is the mean first entrance time from state $i$ to state $j$, and $m^e_{j}$ is the mean first entrance time to state $j$ from the stationary distribution, that is,
\[
m_{i,j}:= \Expect \big[ \inf\{t \,:\, Q(t) = j  \} ~|~ Q(0)=i \big],
\qquad
m^e_{j}: = \sum_{i=0}^K \pi_i m_{i,j},
\] see \cite{CoolenSc}.
Note that we take the indices of the matrices and vectors of size $K+1$ used here to run on range $\{0,\ldots,K\}$.

\subsection{M/M/1/K queue: Explicit formulas related to the Drazin inverse}

As will be evident below, we are particularly interested in the bottom left entry of the Drazin inverse and that of its square,
\begin{align*}
\overline{d}_v &
:=\Lambda^{\sharp}_{K,0} = \int_0^\infty \left(\Prob \big(Q(t)=0 ~|~Q(0)=K\big) - \pi_0\right) \, dt
=\pi_0\,(m^e_{0}-m_{K,0}),
\\
\overline{d}_b &
:=\left(\Lambda^{\sharp}\Lambda^{\sharp}\right)_{K,0}= \pi_0 \sum_{j=0}^K\pi_j\,(m^e_{j}-m_{K,j})\,(m^e_{0}-m_{j,0}).
\end{align*}

Finding explicit expressions for these quantities is tedious yet possible for the M/M/1/K queue:
\begin{lemma}\label{lemdraz}
For the M/M/1/K queue length continuous-time Markov chain, the bottom left elements of the Drazin inverse and its square are
\begin{align*}
\overline{d_v}&=\left\{\begin{array}{ll} 
-\mu^{-1}\,
\dfrac{
K(1-\rho)(1+\rho^{K+1})-2\rho(1-\rho^K)
}{(1-\rho)\,(1-\rho^{K+1})^2}, & \rho\neq 1,\\[1em]
-\mu^{-1}\,
\dfrac{K(K+2)}{6(K+1)} , & \rho=1.\end{array}\right.
\\[1em] 
\overline{d}_b&= 
\left\{\begin{array}{ll}
-\mu^{-2}\, \left\{
\dfrac{
\left[6(1+\rho^2)(1+K)^2-4\rho(1+6K+3K^2)\right] \rho^{K+1}+\rho^2(2+3K+K^2)(1+\rho^{2K})}
{2(1-\rho)^3(1-\rho^{K+1})^3}\right.
&\\[1em] \phantom{-\mu^{-2}\;}\left.-\dfrac{
2\rho(3+2K+K^2)(1+\rho^{2(K+1)})
-K(1+K)(1+\rho^{2(K+2)})}
{2(1-\rho)^3(1-\rho^{K+1})^3}\right\},
& \rho\neq 1,\\[1em]
-\mu^{-2}\,
  \dfrac{7K^4+28K^3+37K^2+18K}{360\,(K+1)},
    & \rho=1.\end{array}\right.
\end{align*}
\end{lemma}
\begin{proof}
For the M/M/1/K queue,
a standard application of ``first step analysis'' leads to the following recurrence equations for $m_{i,j}$:
$$m_{i,j}=\left\{\begin{array}{ll}
0,
& i=j,\\
\lambda^{-1}+m_{1,j},
& i=0,j \neq i,\\
\mu^{-1}+m_{K-1,j},
&
i=K, j \neq i,\\
(\lambda+\mu)^{-1}+\dfrac{\lambda}{\lambda + \mu}m_{i+1,j}+\dfrac{\mu}{\lambda + \mu}m_{i-1,j},
&
\mbox{otherwise}.
 \end{array}\right.
$$ 
When $\rho\neq 1$, the solution is
\begin{align*}
 m_{i,j}&
 =\left\{\begin{array}{ll} 
 \mu^{-1} \Big(
 \dfrac{\rho^{-j}-\rho^{-i}}{(1-\rho)^2} + \dfrac{i-j}{1-\rho}
 \Big),
  & 0\leq i\leq j,\\[1em]
\mu^{-1} \Big(
 \rho^{K+1}\dfrac{\rho^{-j}-\rho^{-i}}{(1-\rho)^2}+ \dfrac{i-j}{1-\rho}
 \Big),
  & j\leq i\leq K,\end{array}\right.
 \end{align*}
 and, when $\rho=1$, the solution is
 \begin{align*}
 m_{i,j}&
 =\left\{\begin{array}{ll}
 \mu^{-1}\,
  \dfrac{j\,(j+1)-i\,(i+1)}{2}, 
  & 0\leq i\leq j,
  \\[1em]
  \mu^{-1}\,
  \dfrac{(K-j)\,\big[(K-j)+1\big]-(K-i)\,\big[(K-i)+1\big]}{2}, 
 & j\leq i\leq K.\end{array}\right.
 \end{align*}
Averaging over $\vc \pi$ we get, 
\[
m^e_j=\left\{\begin{array}{ll}
\mu^{-1}\,
\dfrac{
\rho^{-j}-(1+2j)(1-\rho) - [1+2(K-j)](1-\rho)\rho^{K+1}-\rho^{2(K+1)-j}
}
{
(1-\rho)^2 (1-\rho^{K+1})
}
, & \rho \neq 1,\\[1em]
\mu^{-1}\,
\Big(
j^2-K\,j+\dfrac{K^2}{3}+\dfrac{K}{6}
\Big), & \rho =1.
\end{array}\right.
\]
Combining the above we get the desired results.
\end{proof}

\subsection{M/M/1/K queue: The stationary case}

When the queue length process $\{Q(t)\}$ is stationary, the MAP $\{D(t)\}$ is a (time) stationary point process (see \cite{bookAsmussen2003}). In this case, the asymptotic variance rate, $\bar{v}$, and the y-intercept, $\bar{b}_e$, are respectively given by
 \begin{eqnarray}
 \label{vv1}
 \bar{v}&=&\vc\pi \Lambda_1\vc1-2(\vc\pi \Lambda_1\vc1)^2+2\vc\pi \Lambda_1 \Lambda^- \Lambda_1 \vc1, \\
 \label{bb1}
 \bar{b}_e&=&2(\vc\pi \Lambda_1\vc1)^2-2\vc\pi \Lambda_1 \Lambda^-\Lambda^- \Lambda_1\vc 1,
\end{eqnarray}
see for instance \cite{NarayanaNeuts92}, \cite{bookAsmussen2003} or the summary in \cite{NazarathyWeiss0336}.
By substituting \eqref{rel} into (\ref{vv1},\ref{bb1}) we obtain a simpler expression for $\bar{v}$ and $\bar{b}_e$ in terms of the Drazin inverse:
\begin{eqnarray}
\label{v1}
\bar{v}&=&\vc\pi \Lambda_1\vc1+2\vc\pi \Lambda_1 \Lambda^{\sharp} \Lambda_1 \vc1 
=\lambda^* + 2\vc\pi \Lambda_1 \Lambda^{\sharp} \Lambda_1 \vc1,
\\
\label{b1}
 \bar{b}_e&=&-2\vc\pi \Lambda_1 \Lambda^{\sharp}\Lambda^{\sharp} \Lambda_1\vc 1,
 \end{eqnarray} 
 where
 \[
 \lambda^*= \lim_{t \to \infty} \frac{\Expect[D(t)]}{t} = \vc\pi \Lambda_1 \vc 1=\left\{\begin{array}{ll}
 \mu\rho  \dfrac{1-\rho^ K}{1-\rho^ {K+1}},
 & \textrm{for }\rho\neq 1,\\[1em]
 \mu\, \frac{K}{K+1}, &\textrm{for }\rho= 1.\end{array}\right. 
 \]
We can go further in the simplification of the expressions by using \eqref{p1} and \eqref{p2}. Let $\vc e_i$ denote the column vector of which the only nonzero entry is the entry corresponding to state $i$, which is equal to 1 ($0\leq i\leq K$). First, observe that $\vc\pi \Lambda_1$ and $\Lambda_1 \vc1$ take simple forms in our case:
 \[
\vc\pi \Lambda_1=\left\{\begin{array}{ll}
 \mu\rho\vc\pi-\rho^{K+1}\dfrac{1-\rho}{1-\rho^{K+1}} \mu \vc e_{K}^{'},
 & \textrm{for }\rho\neq 1,\\[1em]
 \mu\vc \pi-\dfrac{1}{K+1} \mu \vc e_{K}^{'},
 &\textrm{for }\rho= 1.\end{array}\right.
,
\qquad\qquad
\Lambda_1\vc 1 = \mu(\vc 1- \vc e_{0}).
\]
Then, since $\vc\pi$ and $\vc1$ are respectively the left and right eigenvectors of $\Lambda^{\sharp}$  corresponding to the eigenvalue 0, we obtain
\[
 \vc{\pi}\Lambda_1 \Lambda^{\sharp}\,\Lambda_1 \vc1
 =
 \left\{\begin{array}{ll}
 \label{eq1}
\mu^ 2\rho^{K+1}\dfrac{1-\rho}{1-\rho^{K+1}}\, \overline{d}_v,
 & \textrm{for }\rho\neq 1,\\[1em]
\mu^ 2\dfrac{1}{K+1}\,\overline{d}_v, &\textrm{for }\rho= 1.
\end{array}\right. 
 \]
We thus obtain
$$\bar{v}=
\left\{\begin{array}{ll}
\lambda^* + 2\mu^2\rho^{K+1}\dfrac{1-\rho}{1-\rho^{K+1}} \overline{d}_v
,
 & \rho\neq 1,\\[1em] 
 \lambda^* + \dfrac{2\mu^2}{K+1}\overline{d}_v,& \rho = 1,\end{array}\right.
 $$
and similarly,
$$\bar{b}_e=
\left\{\begin{array}{ll}
 -2 \mu^2 \, \rho ^{K+1}\dfrac{1-\rho}{1-\rho ^{K+1}} \overline{d}_b\,,
& \rho\neq 1,\\[1em] 
-2  \mu^2  \, \dfrac{1}{K+1} \overline{d}_b\,,
& \rho = 1.\end{array}\right. 
$$ 
\begin{figure}[t] 
\centering
\includegraphics[angle=0, width=12cm]{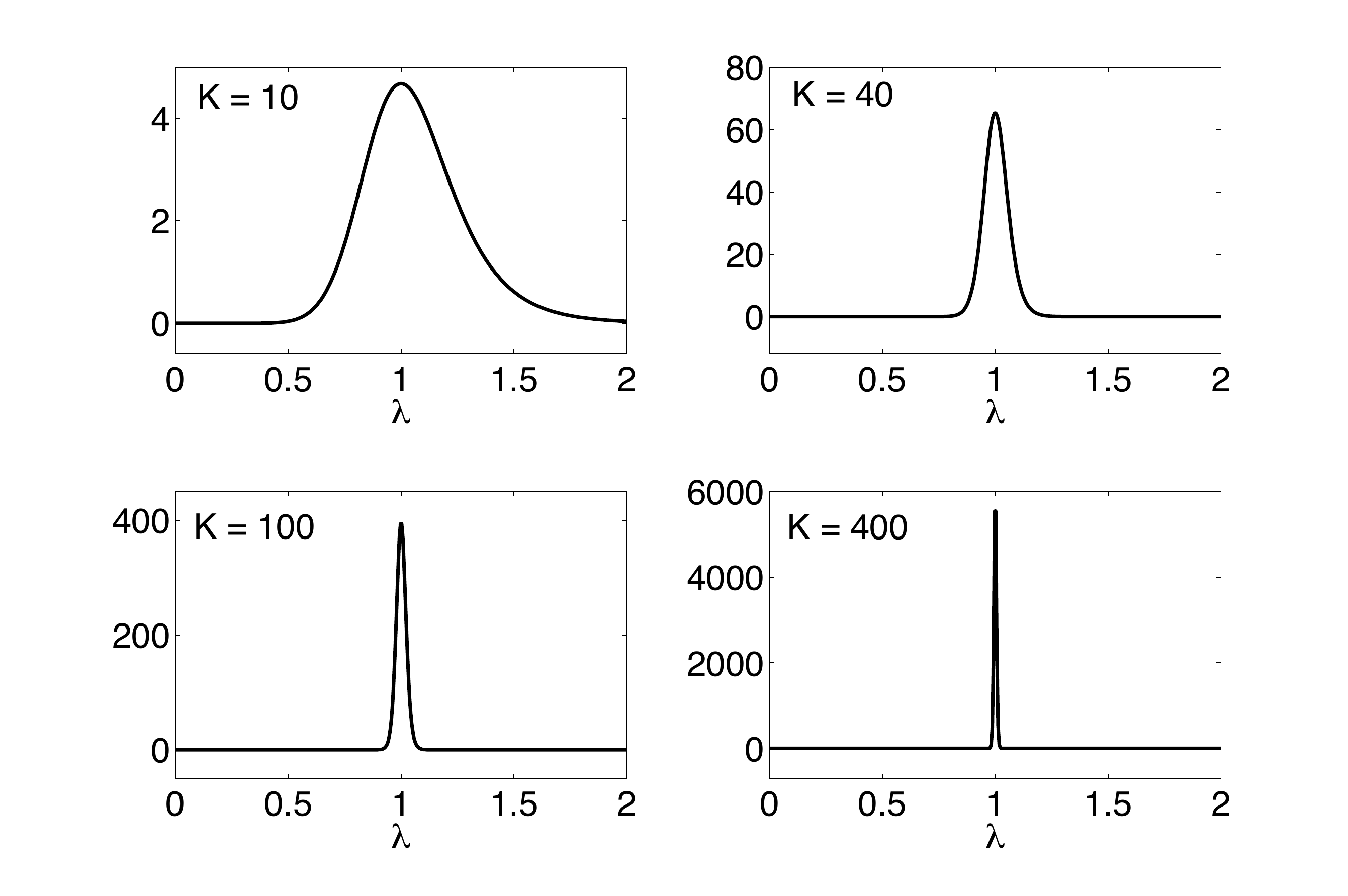} 
\caption{\label{yint}The y-intercept $\bar{b}_e$ as a function of $\lambda$ when $\mu=1$ and for $K=10,40,100,400$.}
 \end{figure}
Combining the above with the results of Lemma~\ref{lemdraz}, and manipulating the expressions, we obtain our main result for M/M/1/K queues:
\begin{proposition} \label{propvbe}
For the stationary M/M/1/K queue, $v(t)=\bar{v} t+\bar{b}_e+o(1)$ where the asymptotic variance rate and y-intercept are
\begin{align*}
\bar{v}&=\left\{\begin{array}{ll} 
\lambda \,
\dfrac{(1+\rho^{K+1})\big(1-(1+2K)\rho^K(1-\rho)-\rho^{2K+1}\big)}{(1-\rho^{K+1})^3}, 
& 
\rho\neq 1\\[1em]
\lambda \,\Big(\dfrac{2}{3}-
 \dfrac{3K+2}{3(K+1)^2}\Big), & \rho=1\end{array}\right.,\\
&\mbox{and}\\
\bar{b}_e&=\left\{\begin{array}{ll} 
\rho^ {K+1}\, \left\{
\dfrac{
\left(6(1+\rho^2)(1+K)^2-4\rho(1+6K+3K^2)\right) \rho^{K+1}+\rho^2(2+3K+K^2)(1+\rho^{2K})}
{(1-\rho )^2 \left(1-\rho ^{K+1}\right)^4}\right.
&\\[1em] \phantom{\rho^ {K+1}\;}\left.-\dfrac{
2\rho(3+2K+K^2)(1+\rho^{2(K+1)})
-K(1+K)(1+\rho^{2(K+2)})}
{(1-\rho )^2 \left(1-\rho ^{K+1}\right)^4}\right\},
&
\rho \neq 1,
\\[1em]
 \dfrac{7K^4+28K^3+37K^2+18K}{180\,(K+1)^2}, 
 & \rho=1.
 \end{array}\right.,
 \end{align*}
respectively.\hfill $\blacksquare$
\end{proposition}
\medskip
\noindent
Here are some observations:
\begin{itemize}
\item
With the exception of $\overline{b}_e$ for $\rho \neq 1$, all of the expressions in Proposition~\ref{propvbe} appeared previously in \cite{NazarathyWeiss0336}. Yet, while working on \cite{NazarathyWeiss0336}, the authors were not able to obtain $\overline{b}_e$ when $\rho \neq 1$, as is obtained now. 

\item We illustrate the y-intercept for different values of $K$ and $\lambda$ in Figure~\ref{yint}. It is straightforward to see that
\[
\lim_{K \to \infty} \overline{b}_e = 
\left\{\begin{array}{ll} 
0, & \rho \neq 1,\\
\infty, &  \rho=1,
\end{array}\right.
\]
and further, for $\rho=1$, $\overline{b}_e = O(K^2)$.
\item 
It is insightful to see the role of $\overline{d}_v$ and $\overline{d}_b$ in the above derivations.  In fact, the spikes in $\overline{v}$ and $\overline{b}$ that occur at $\rho \approx 1$ are attributed to $\overline{d}_v$ and $\overline{d}_b$.
\end{itemize}

\subsection{Some further useful results on MAPs} 
 
Our derivation of $\overline{v}$ and $\overline{b}_e$ above is based on \eqref{vv1} and \eqref{bb1} respectively, or alternatively on their Drazin inverse based forms, \eqref{v1} and \eqref{b1}. 
On route to calculating additional performance measures for the M/M/1/K queue, we first derive some further MAP results, which to the best of our knowledge have not appeared elsewhere. These results are of independent interest.

Consider an arbitrary MAP with an $n \times n$ irreducible generator matrix $\Lambda=
 \Lambda_0 + \Lambda_1$, where $\Lambda_1$ is the event intensity matrix and $\Lambda_0$ is assumed to be non-singular. Such a MAP corresponds to a two-dimensional Markov chain $\{(N(t),\varphi(t)),t\geq 0\}$, where $N(t)$ denotes the number of events in the interval $[0,t]$ and is also called the \textit{level} of the MAP at time $t$, and $\varphi(t)$ denotes the {\em phase} at time $t$, taking values in $\{1,\ldots,n\}$. We assume that $N(0)=0$ almost surely and denote by $\vc \theta$ the distribution of $\varphi(0)$. Further, we denote by $\vc \pi$ the stationary distribution corresponding to $\Lambda$, or equivalently to the phase process $\{\varphi(t)\}$.
 
If $\vc \theta = \vc \pi$ then $N(t)$ is a time-stationary point process. This implies that for any sequence of intervals $(t_1,s_1),\ldots,(t_\ell,s_\ell)$ and for any $\tau$,
\[
\big[N(s_1)-N(t_1),\ldots,N(s_\ell)-N(t_\ell)\big]
=^d 
\big[N(s_1+\tau)-N(t_1+\tau),\ldots,N(s_\ell+\tau)-N(t_\ell+\tau)\big],
\]
where the equality is in distribution (see \cite{bookAsmussen2003}, Chapter XI, Proposition 1.2).

Another interesting initial distribution is $\vc\alpha :=\vc \pi \Lambda_1/(\vc\pi\Lambda_1\vc1)$. This is the invariant distribution of a discrete time jump chain, $\{Y_k\}$ where $Y_k$ is the value of $\varphi(t)$ just after the $k$th arrival. The probability transition matrix of this Markov chain is $-\Lambda_0^{-1}\Lambda_1$. As shown in \cite{bookAsmussen2003}, Chapter XI, Proposition 1.4, setting $\vc \theta = \vc \alpha$ makes $\{N(t)\}$ an event-stationary point process. That is, if $T_k$ denotes the time interval between the $(k-1)st$ and the $k$th event in the MAP, then the joint distribution of $(T_k,T_{k+1},\ldots,T_{k+\ell})$ is the same as the joint distribution of $(T_{k'},T_{k'+1},\ldots,T_{k'+\ell})$ for all integer $k,k',\ell$.

Since $\vc \theta$ affects the point process in such a manner, it is natural to see its effect on $v(t)$ and related quantities. We now have the following:
\begin{proposition}\label{yintarbmap} For an arbitrary MAP with initial distribution $\vc\theta$,
$\Var\big(N(t)\big)=\bar{v} t+\bar{b}_\theta+o(1)$,
where the y-intercept is given by
\begin{equation}
\label{b0}\bar{b}_\theta=\bar{b}_e -\left(2\vc\pi \Lambda_1\vc1\vc\theta(\Lambda^{\sharp})^2\Lambda_1\vc 1-2\vc\theta \Lambda^{\sharp}\Lambda_1\Lambda^{\sharp}\Lambda_1\vc1+(\vc\theta \Lambda^{\sharp}\Lambda_1\vc1)^2\right),\end{equation} and $\bar{v}$ and $\bar{b}_e$ are respectively given by \eqref{v1} and \eqref{b1}.
\end{proposition}
\begin{proof} 
The variance curve of an arbitrary MAP with initial distribution $\vc\theta$ is given by
\begin{equation}\label{varmap}\Var\big(N(t)\big)=\vc\theta M_2(t)\vc1+\vc\theta M_1(t)\vc1-(\vc\theta M_1(t)\vc1)^2,\end{equation}where
$M_1(t)$ and $M_2(t)$ denote the matrices of the first two factorial moments of the number of events in a non-stationary MAP, that is,
$$\begin{array}{rcl} [M_1(t)]_{ij}&=&\Expect[N(t)\,\mathds{1}_{\{\varphi(t)=j\}}| \varphi(0)=i]\\[1em] [M_2(t)]_{ij}
&=&\Expect[N(t)(N(t)-1)\,\mathds{1}_{\{\varphi(t)=j\}}| \varphi(0)=i].
\end{array}$$
Narayana and Neuts \cite{NarayanaNeuts92} showed that $M_1(t)$ has a linear asymptote in that there exist constant matrices $A_0$ and $A_1$ such that
\begin{equation}\label{m1t}M_1(t)=A_0t+A_1+O(e^{-\eta t} t^{2r-1})\quad \mbox{as } t\rightarrow \infty,\end{equation}
where $-\eta$ is the real part of $\eta^*$, the non-zero eigenvalue of $\Lambda$ with maximum real part, and $r$ is the multiplicity of $\eta^*$. Similarly, $M_2(t)$ has a quadratic asymptote in that there exist constant matrices $B_0, B_1$ and $B_2$ such that
\begin{equation}\label{m2t}M_2(t)=B_0t^2+2B_1t+2B_2+O(e^{-\eta t} t^{3r-1})\quad \mbox{as } t\rightarrow \infty.\end{equation} The expressions for the coefficient matrices $A_0,\,A_1,\,B_0,\,B_1,$ and $B_2$ given in \cite{NarayanaNeuts92} are in terms of the fundamental matrix $\Lambda^-$. After rewriting them in terms of the deviation matrix via the relation \eqref{rel}, we obtain the following matrices:
\begin{eqnarray*} A_0 &=& (\vc\pi \Lambda_1 \vc 1)\,\vc 1\vc\pi
\\ A_1 &=& \Lambda^{\sharp} \Lambda_1\vc 1\vc\pi+\vc1\vc\pi \Lambda_1 \Lambda^{\sharp}
\\ B_0&=& (\vc\pi \Lambda_1 \vc 1)^2\,\vc 1\vc\pi
\\ B_1&=&  (\vc\pi \Lambda_1 \vc 1)\,\Lambda^{\sharp}\Lambda_1\vc 1\vc\pi+(\vc\pi \Lambda_1 \vc 1)\,\vc1\vc\pi\Lambda_1\Lambda^{\sharp}+ (\vc\pi\Lambda_1 \Lambda^{\sharp} \Lambda_1\vc1)\,\vc1\vc\pi
\\B_2&=&-\vc1\vc\pi (\vc\pi\Lambda_1(\Lambda^{\sharp} )^2\Lambda_1\vc1)+\Lambda^{\sharp} \Lambda_1\vc1\vc\pi\Lambda_1\Lambda^{\sharp} -(\vc\pi\Lambda_1\vc1)(\Lambda^{\sharp} )^2\Lambda_1\vc1\vc\pi\\&&-(\vc\pi\Lambda_1\vc1)\vc1\vc\pi\Lambda_1(\Lambda^{\sharp} )^2 +\vc1\vc\pi\Lambda_1\Lambda^{\sharp} \Lambda_1\Lambda^{\sharp}+\Lambda^{\sharp}\Lambda_1\Lambda^{\sharp}\Lambda_1\vc1\vc\pi.
\end{eqnarray*} By injecting \eqref{m1t} and \eqref{m2t} in terms of $\Lambda^{\sharp}$ into \eqref{varmap}, and after a few simplifications, we obtain that $\Var\big(N(t)\big)$ has a linear asymptote whose y-intercept is 
$$\bar{b}_\theta=-2\vc\pi \Lambda_1 \Lambda^{\sharp}\Lambda^{\sharp} \Lambda_1\vc 1 -2\vc\pi \Lambda_1\vc1\vc\theta(\Lambda^{\sharp})^2\Lambda_1\vc 1+2\vc\theta \Lambda^{\sharp}\Lambda_1\Lambda^{\sharp}\Lambda_1\vc1-(\vc\theta \Lambda^{\sharp}\Lambda_1\vc1)^2,$$
which proves the theorem.
\vspace{10pt}
\end{proof}
\medskip
\noindent

Here are some observations:
\begin{itemize}
\item
For $\vc\theta=\vc\pi$, the correction term $\bar{b}_\theta-\bar{b}_e$ vanishes as expected.
\item 
The correction term does not depend only on the variance of the initial distribution $\vc\theta$, as we show to be the case for the M/G/1 queue (see Proposition~\ref{prop:coupling}). Indeed, consider $n=3$, $\vc\theta_1=[1,0,0]$ and $\vc\theta_2=[0,1,0]$ which have the same variance equal to zero; however it is easy to find an example of a MAP for which the value of $\bar{b}_\theta-\bar{b}_e$ is different for the two initial distributions.  
\item In the specific case where $\Lambda_1\vc1=\beta\vc1$ (where $\beta$ is a constant),  the correction term $\bar{b}_\theta-\bar{b}_e$ vanishes for all initial distributions $\vc\theta$ because of the property \eqref{p1}.  This shows that $\vc\theta=\vc\pi$ is a sufficient but not necessary condition for having $\bar{b}_\theta=\bar{b}_e$.
\end{itemize}
 
A further performance measure of interest is the asymptotic covariance between the level and the phase of a MAP. As shown in the following proposition, the Drazin inverse also plays a role in that asymptotic quantity.

\begin{proposition}\label{covstat} 
Let $\{N(t)\}$ and $\{\varphi(t)\}$  be the level and the phase processes of a MAP with initial phase distribution $\vc \theta$. Then,
\[
\lim_{t\rightarrow\infty} \Cov\big(N(t),\varphi(t)\big)=
\sum_{i=1}^{n} i \big(\vc\pi \Lambda_1 \Lambda^{\sharp}\big)_i
-\left(\sum_{i=1}^{n} i\,\pi_i\right)\, \vc\theta \Lambda^{\sharp} \Lambda_1 \vc 1.
\]
Further, in the time-stationary case ($\vc \theta = \vc \pi$), the term with the second sum vanishes.
\end{proposition}
\begin{proof} 
For a MAP with initial distribution $\vc\theta$, define $M^{\theta}_i(t)=\Expect[N(t)\,\mathds{1}_{\{\varphi(t)=i\}}]$, and $\vc M^{\theta}(t)=[M^{\theta}_1(t),\ldots,M^{\theta}_n(t)]$. From \cite[Chapter XI, Proposition 1.7]{bookAsmussen2003}, we have
$$\vc M^{\theta}(t)= \vc\theta\int_0^t e^{\Lambda u}\,\Lambda_1\,e^{\Lambda (t-u)}\, du.$$
Let us define the \textit{transient} Drazin inverse as
$$\Lambda^{\sharp}(t)=\int_0^t \left(e^{\Lambda u}-\vc 1\vc \pi\right)\,du=\int_0^t e^{\Lambda u}\,du- \vc 1\vc \pi\,t,	$$ so that $\Lambda^{\sharp}=\lim_{t\rightarrow\infty} \Lambda^{\sharp}(t)$. 
 With this, 
$$\Expect[N(t)]=\vc M^{\theta}(t)\vc 1=\vc\theta \int_0^te^{\Lambda u}du\,\Lambda_1\vc1=\vc \pi \Lambda_1\vc 1 t+ \vc\theta \Lambda^{\sharp}(t) \Lambda_1\vc 1.$$ Next, note that
$\Expect[N(t)\,\varphi(t)]=\sum_{i=1}^n i\,M^{\theta}_i(t)$. Therefore, since
$\Expect[\varphi(t)]=\sum_{i=1}^n i\,(\vc\theta e^{\Lambda t})_i$, we obtain 
\begin{eqnarray*}
\Cov\big(N(t),\varphi(t)\big)
&=&
\Expect[N(t)\,\varphi(t)]-\Expect[N(t)]\,\Expect[\varphi(t)]\\&=&
\sum_{i=1}^n i \left\{\left[\vc\theta e^{\Lambda t}\int_0^te^{-\Lambda u}\Lambda_1 e^{\Lambda u}du\right]_i-\left(\vc\pi \Lambda_1\vc1 t+\vc\theta \Lambda^{\sharp}(t) \Lambda_1\vc1\right)\,\left[\vc\theta e^{\Lambda t}\right]_i\right\}.
\end{eqnarray*} 
Finally, we take $t\rightarrow\infty$ in the last expression, and we use the fact that $\lim_{t\rightarrow\infty} \vc\theta e^{\Lambda t}=\vc\pi,$ $\vc\pi e^{-\Lambda u}=\vc\pi$, and $\int_0^te^{\Lambda u}du=\Lambda^{\sharp}(t)+\vc 1\vc\pi t$. After some algebraic simplifications, we obtain the statement of the proposition.
\end{proof}

 \begin{figure}[t] 
\centering
\includegraphics[angle=0, width=12cm]{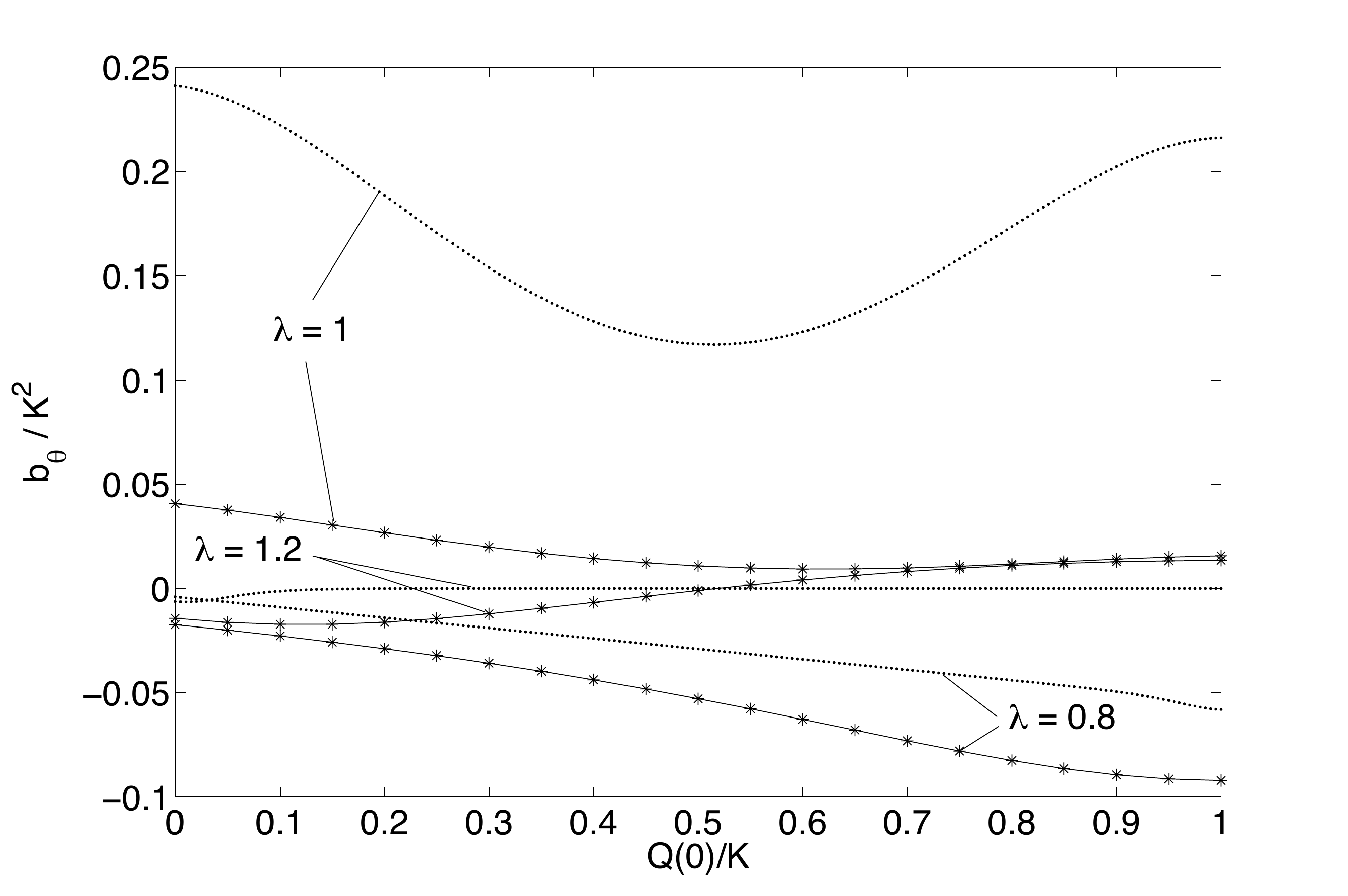} 
\caption{\label{fig1}The scaled y-intercept $\overline{b}_\theta/K^2$ when $\mu=1$ and $\vc\theta=\vc e_i'$ (that is, $Q(0)=i$ almost surely) as a function of $i/K\in\{0,1/K,2/K,\ldots,1\}$, for $\lambda=0.8$, $\lambda=1$, and $\lambda=1.2$, for $K=20$ (stars) and $K=200$ (dots).}
 \end{figure}

\subsection{M/M/1/K queue: Arbitrary initial conditions}

We now make use of Proposition \ref{yintarbmap} to investigate the y-intercept $\bar{b}_{\theta}$ of $v(t)$ for an arbitrary M/M/1/K queue where the distribution of $Q(0)$ is $\vc\theta$.

In Figure \ref{fig1}, we show the (scaled) value of $\overline{b}_\theta$ for the particular initial distributions $\vc\theta=\vc e_i'$, that is, for M/M/1/K queues which start with $i$ customers at time $t=0$ almost surely, for $0\leq i\leq K$. We observe that when $\rho<1$, $\overline{b}_\theta$ is a monotonically decreasing function of $i$, while when $\rho\geq 1$, $\overline{b}_\theta$ exhibits a minimum. 
We also observe that for $\rho>1$, when $K$ increases and $Q(0)/K\rightarrow 1,$ 
$\overline{b}_\theta\rightarrow 0$.

In Figure \ref{fig4}, we consider the behaviour of the y-intercept $\overline{b}_\theta$ in the event-stationary case, that is, when $\vc\theta=\vc \alpha$. The four graphs show the correction term $\overline{b}_\theta-\overline{b}_e$ as a function of $\rho$ (or, more precisely, as a function of $\lambda$ for a fixed value of $\mu$) for increasing values of $K$. We see that the curves have an interesting shape with two local minima centered around $\rho=1$. As $K$ increases, the dips become narrow and deep; the correction term converges to zero everywhere, except at $\rho=1$ where further computation has shown that it decreases approximately linearly. As a consequence, the effect of event-stationarity on the y-intercept becomes indistinguishable from the effect of time-stationarity as $K\rightarrow\infty$ for all values of $\rho$ except in the balanced case.

Note that 	the explicit expressions for $\overline{b}_\theta$ when $\vc\theta=\vc e_i'$ or $\vc\theta=\vc \alpha$ can also be obtained, but as they are quite cumbersome and do not bring more information, we do not present these here.
 
 \begin{figure}[t] 
\centering
\includegraphics[angle=0, width=12cm]{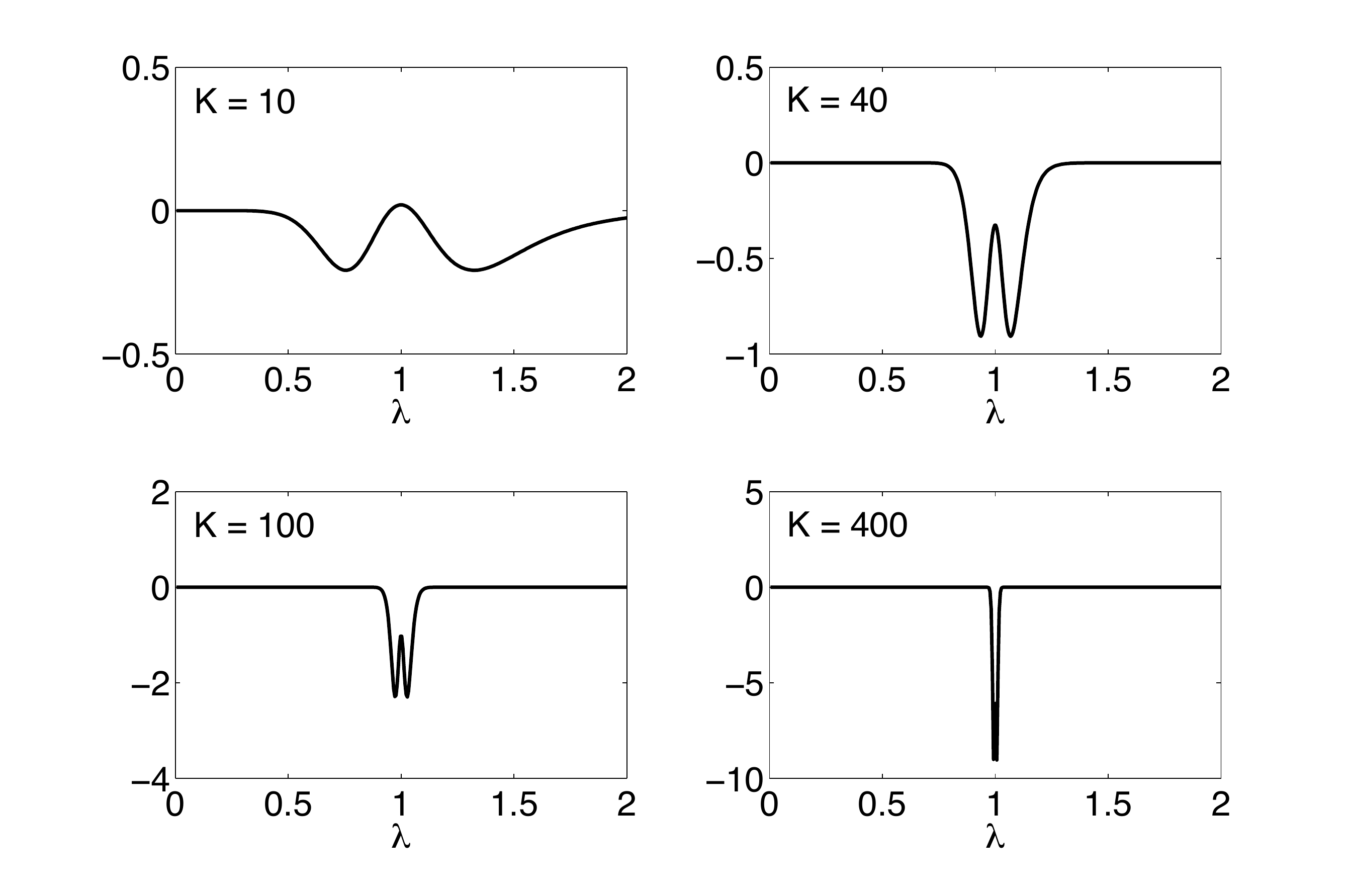} 
\caption{\label{fig4}The correction term $\overline{b}_\theta-\overline{b}_e$ when $\vc\theta=\vc \alpha$ as a function of $\lambda$ when $\mu=1$ and for $K=10,40,100,400$.}
 \end{figure}
 
\subsection{M/M/1/K queue: Asymptotic covariance}

Making use of Proposition \ref{covstat} and the particular structure of the vector $\vc\pi$ and the matrix $\Lambda_1$, in the stationary case, we can express the asymptotic covariance between the number of departures and the queue size explicitly.
\begin{figure}[t] 
\centering
\includegraphics[angle=0, width=12cm]{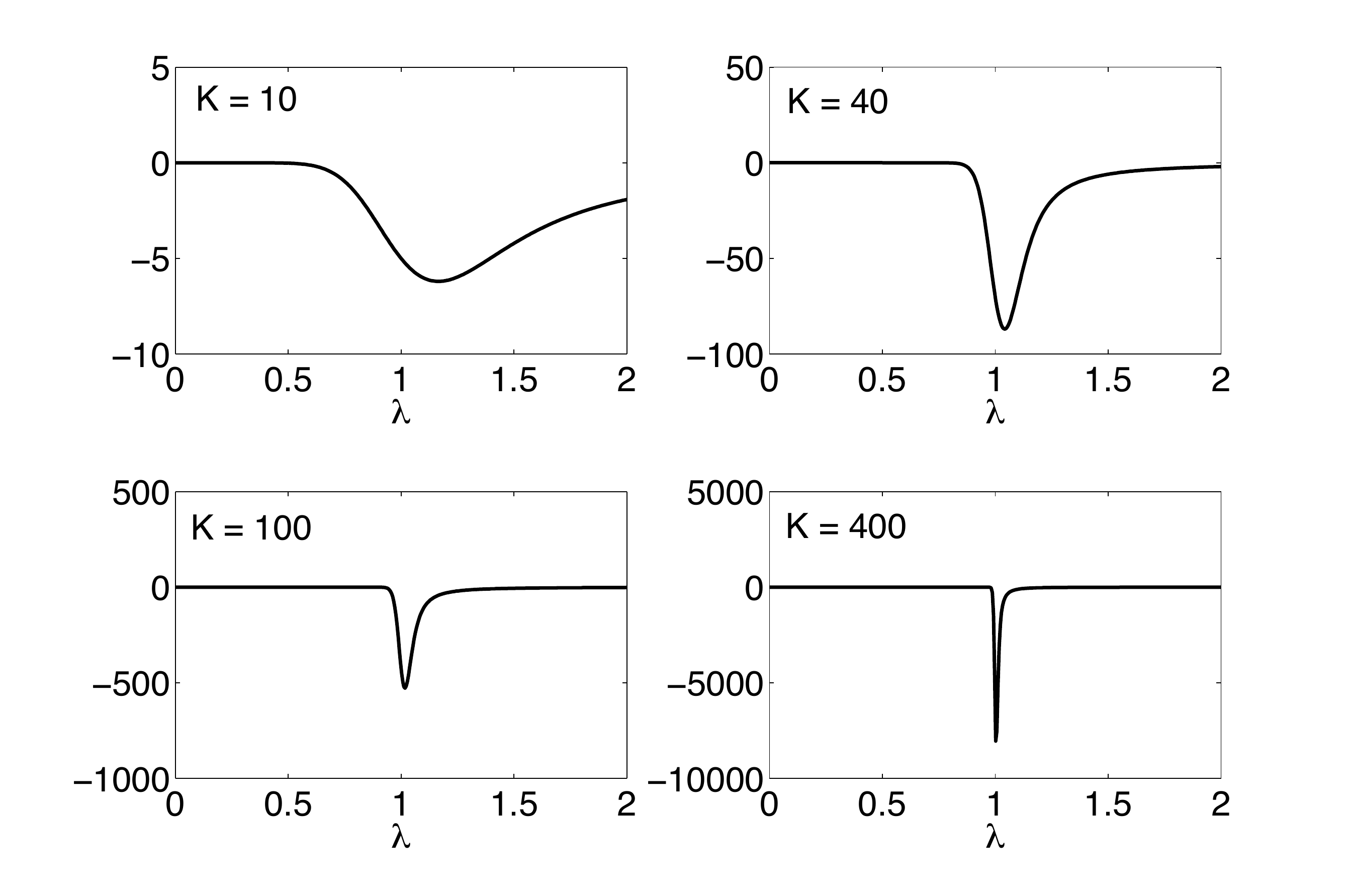} 
\caption{\label{fig6}The asymptotic covariance between $D(t)$ and $Q(t)$ as a function of $\lambda$ when $\mu=1$ and for $K=10,40,100,400.$}
\end{figure}
\begin{corollary} Consider the stationary M/M/1/K queue with output process $\{D(t)\}$ and queue level process $\{Q(t)\}$. Then, 
\begin{align*}
\lim_{t\rightarrow\infty} \Cov(D(t),Q(t))&=\left\{\begin{array}{ll} 
\rho^{K+1}\left\{\dfrac{ K^2 (\rho-1)^2 \left(1+3 \rho^{K+1}\right)-2 \rho \left(\rho^K-1\right) \left(-2+\rho+\rho^{K+2}\right)}{2 (\rho-1)^2 \left(\rho^{K+1}-1\right)^3}\right.
&\\[1em] \phantom{\rho^ {K+1}\;}\left.+\dfrac{K (\rho-1) \left(-1+3 \rho-7 \rho^{K+1}+5 \rho^{K+2}\right)}{2 (\rho-1)^2 \left(\rho^{K+1}-1\right)^3}\right\},
&
\rho \neq 1,
\\[1em]
 -K(K+2)/24, 
 & \rho=1.
 \end{array}\right.
 \end{align*}

\end{corollary}
\begin{proof}We use Proposition \ref{covstat} with $N(t)=D(t)$ and $\varphi(t)=Q(t)$, together with the fact that $(\vc \pi \Lambda_1 \Lambda^{\sharp})_i= -C\,\Lambda^{\sharp}_{Ki},$ where 
$$C=\left\{\begin{array}{ll} \mu\dfrac{\rho^{K+1}(1-\rho)}{1-\rho^{K+1}},& \textrm{for }\rho\neq 1,\\[1em] \dfrac{\mu}{K+1},&\textrm{for }\rho= 1,\end{array}\right.$$ and $\Lambda^{\sharp}_{Ki}=\pi_i\,(m^e_{i}-m_{K,i}).$ The entries $\Lambda^{\sharp}_{Ki}$ of the Drazin inverse are then computed explicitly for $0\leq i\leq K$ using the expressions for $\vc\pi$ and $m_{i,j}$ derived in the proof of Lemma \ref{lemdraz}.
\end{proof}

\medskip

Note that Proposition \ref{covstat} indicates that the difference between the stationary and the non-stationary cases is the correction term $-(\sum_i i\pi_i) \vc\theta \Lambda^{\sharp} \Lambda_1 \vc 1$, where $\sum_i i\pi_i=K/2$ for $\rho=1$, and for $\rho\neq 1 $,
 $$\sum_{i=0}^K i\pi_i=\frac{\rho \left(1-(1+K) \rho^K+K \rho^{1+K}\right)}{(1-\rho) \left(1-\rho^{1+K}\right)}.$$

In Figure \ref{fig6}, we illustrate the asymptotic covariance between $D(t)$ and $Q(t)$ in the stationary case, as a function of $\rho$ and for increasing values of $K$. We see that the asymptotic covariance curves exhibit a similar behaviour to (the negative of) those of the y-intercept in the stationary case (see Figure \ref{yint}), but in the present case the curves are more skewed with respect to $\rho=1$.

\section{The M/G/1 Queue}
\label{sec:mainResMG1}
We now consider the departure process of the M/G/1 queue. In this case, the departure process $\{D(t)\}$ is generally not a MAP, and the analysis is more complicated. Nevertheless we are able to obtain some partial results about the linear asymptote of  $v(t)$.  Our approach is to first assume the existence of a linear asymptote and then to find an elegant formula for the y-intercept under this assumption, generalizing the y-intercept of the M/M/1 queue in \eqref{eq:mm1Var} for the stable case. Further we conjecture that our assumption holds when the third moment of the service time is finite.

Denote the arrival rate by $\lambda$, the service time distribution by $G(\cdot)$ and its $k$'th moment by $g_k$. In this case, $\mu := g_1^{-1}$, and we assume that $\rho := \lambda/\mu <1$.  The squared coefficient of variation and skewness coefficient are respectively given by
\[
c^2=\frac{g_2}{g_1^2}-1, \qquad \gamma = 	\frac{2 g_1^3 - 3 g_1 g_2 + g_3}{(g_2-g_1^2)^{3/2}}.
\]

\begin{figure}[t] 
\centering
\includegraphics[angle=0, width=12cm]{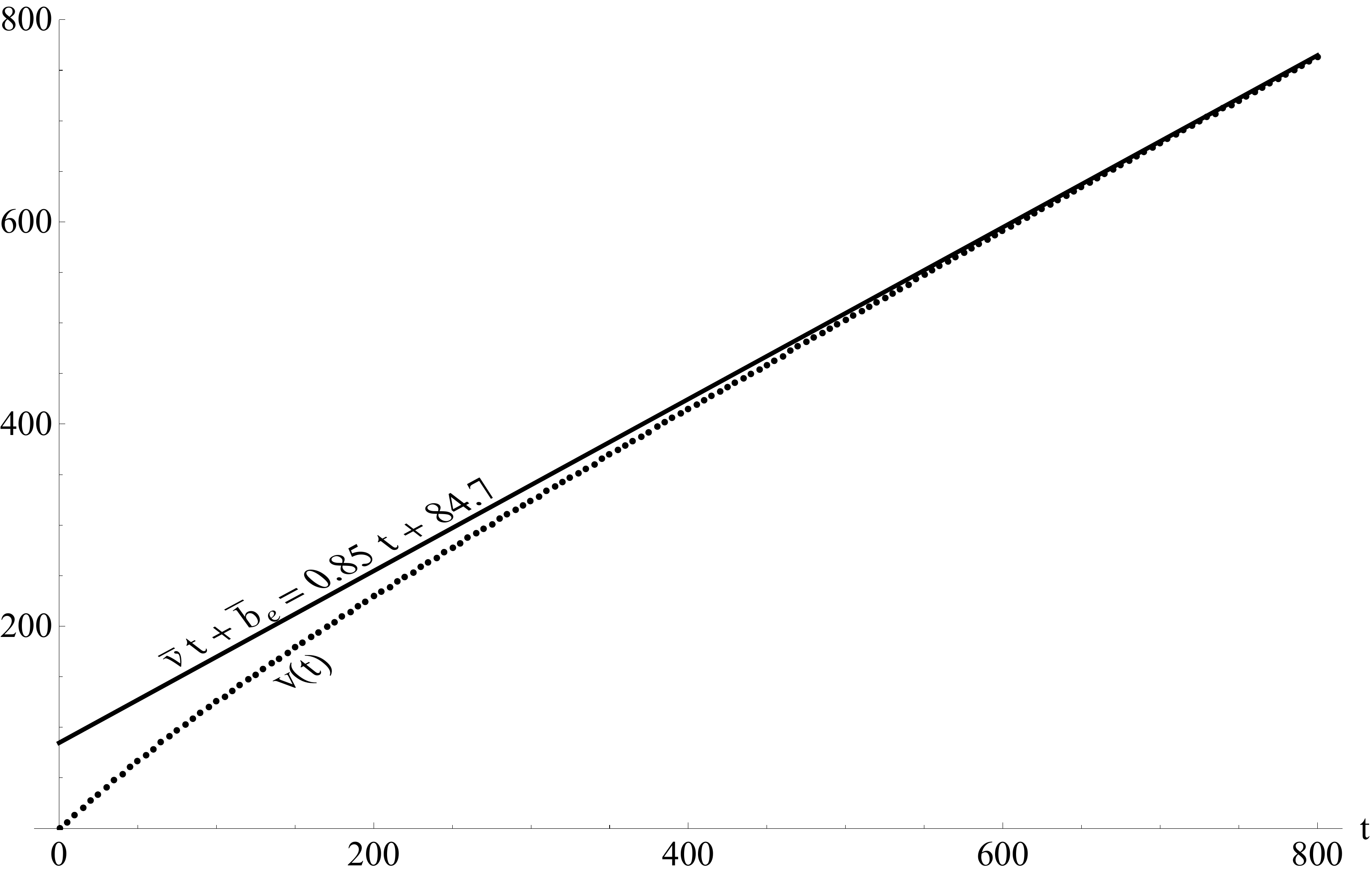} 
\caption{\label{fig:MG1sim}The variance curve and its linear asymptote for a stationary M/G/1 queue with $\lambda = 0.85$, $\mu=1$ and G following a log-normal distribution with $c^2=2$.}
\end{figure}

Consider the numerical example in Figure~\ref{fig:MG1sim}. This is a stationary M/G/1 queue with $\lambda=0.85$ and $\mu=1$ ($\rho=0.85$). The service distribution is taken to be a log-normal distribution with $c^2=2$. This implies $\gamma = 10/\sqrt{2}$. The figure plots the variance curve, $v(t)$, next to the linear asymptote\footnote{We simulated $10^6$ realizations of the queueing process, recording and estimating the variance of $D(t)$ over the grid $t=0,5,10,\ldots,800$. Prior to time $t=0$ we simulated each realization for $3\times 10^4$ units so as to begin in approximate steady-state. The simulation is coded in C to allow for efficient computation. 
During the simulation run, roughly $26 \times 10^{9}$ jobs were processed in the simulated M/G/1 queue.}. It is visually evident that for non-small $t$,
\[
v(t) \approx \overline{v} t + \overline{b}_e.
\]
For the rest of this paper, we shall assume that such a linear asymptote exists. This is stated in the assumption below.
\begin{ass}
\label{sweetAss1}
There exist $\overline{v}$ and $\overline{b}_\theta$ such that,
\begin{equation}\label{ass1}
v(t) = \overline{v} t + \overline{b}_\theta + o(1).
\end{equation}
\end{ass}

The $\overline{b}_e$ term in Figure~\ref{fig:MG1sim} was calculated from the formula in Proposition~\ref{prop:statAsymp} below and is a function of $\mu$, $c^2$ and $\gamma$. As is attested by the figure and by further extensive numerical experiments, we believe that such a term exists for all M/G/1 queues with $\rho \neq 1$ in which the service time distribution has a finite third moment, yet we have not been able to prove this. 
%
%
\begin{conj}
\label{conj1}
Consider an M/G/1 queue with finite third moment operating under any work-conserving non-pre-emptive policy, and assume $\lambda \neq \mu$. Assume that the variance of the queue length at time $0$ is finite, then there exists a finite $\overline{b}_\theta$ such that \eqref{ass1} holds,
with
\[
\overline{v} = 
\left\{\begin{array}{ll}
 \lambda,  &\lambda < \mu, \\
\mu c^2,  & \lambda > \mu.
\end{array}\right.
\]
\end{conj}

To get insight into the asymptotic variance rate, $\overline{v}$, in Conjecture~\ref{conj1}, consider first the case where $\lambda > \mu$. In this case, after some time $\tau$ which is almost surely finite, the server never stops operating and thus for $t > \tau$, $\{D(t)\}$ is effectively a renewal process with asymptotic variance rate $\mu c^2$ (see equation \eqref{eq:renewalAsympVar}). To the best of our knowledge, this intuitive argument cannot be made rigorous easily.

For the case $\lambda < \mu$ consider the relation,
\[
D(t) =A(t)+ Q(0)  - Q(t),
\]
where $D(t)$ is the number of service completions during $[0,t]$, $\{A(t)\}$ is the Poisson arrival process, counting arrivals during $[0,t]$, and $Q(t)$ is the number of customers in the system at time $t$. Taking the variance of both sides of the last equation, observing that $A(t)$ is independent of $Q(0)$, dividing by $t$, and letting $t \to \infty$, we get
\begin{align*}
\lim_{t \to \infty} \frac{v(t)}{t} &= 
\lim_{t \to \infty} \frac{\Var\big(A(t)\big)}{t} +
\lim_{t \to \infty} \frac{\Var\big(Q(0)\big)}{t} +
\lim_{t \to \infty} \frac{\Var\big(Q(t)\big)}{t}\\
&- 2 \lim_{t \to \infty} \frac{\Cov\big(Q(0),Q(t)\big) }{t} -2  \lim_{t \to \infty} \frac{\Cov\big(A(t),Q(t)\big) }{t},
\end{align*}
whenever the limits exist. Now the first limit on the right hand side equals $\lambda$, the second limit vanishes by assumption, the third limit should vanish since the stationary variance is finite (due to a finite third moment -- see equation \eqref{eq:StatPi}), the fourth limit should vanish since in fact $\Cov(Q(0), Q(t))$ vanishes (this is not trivial to establish in general, yet was communicated to us for the FCFS case through personal communication with Brian Fralix), and finally, for the fifth limit observe that
\[
\big| \Cov\big( A(t), Q(t) \big) \big| \le \sqrt{ \lambda t \Var\big(Q(t)\big)} = O(\sqrt{t}),
\]
and thus the limit vanishes. This implies that the departure asymptotic variance equals the arrival asymptotic variance. To get insight into our belief of the importance of $g_3<\infty$ for the existence of $\overline{b}_\theta$ (at least for the case $\lambda < \mu$), see the proof of Proposition~\ref{prop:statAsymp} below. 

Our focus for the rest of this section is on the stable case ($\lambda < \mu$) in which we are able to find explicit expressions for $\overline{b}_e$ and $\overline{b}_\theta$ (including $\overline{b}_0$).

\subsection{M/G/1 queue: The stationary case}
In \cite{daley1975fso} (see also \cite{Daley0510}) Daley found the Laplace-Stieltjes transform (LST) of the variance curve for the stationary case.  After some minor rearrangement, Daley's formula may be written as
\begin{equation}
\label{eq:DaelyFormula}
v^*(s) := \int_0^\infty e^{-st} dv(t) = 
\frac{\lambda}{s} +  b^*(s),
\end{equation}
where,
\begin{equation}
\label{eq:bStarS}
b^*(s)=
\frac{2 \lambda}{s}\left( \frac{G^*(s)}{1-G^*(s)}\left(1 - \frac{s \Pi(\Gamma(s))}{s+\lambda(1-\Gamma(s))}\right) - \frac{\lambda}{s}\right),
\end{equation}
and the LST exists for $\Re(s)>0$. Here, $G^*(\cdot)$ is the Laplace-Stieltjes transform (LST) of $G(\cdot)$, $\Pi(\cdot)$ is the probability generating function of the stationary number of customers in the system, with
\begin{equation}
\label{eq:KPformula}
\Pi(z) = (1-\rho)\frac{(1-z) G^*\big(\lambda(1-z)\big)}
{G^*\big(\lambda(1-z)\big) -z},
\end{equation}
for $\Re(z) \le 1$. Further, $\Gamma(s)$ is the LST of the busy period at $s$, with $\Re(s)>0$. It is obtained as the minimal non-negative solution of
\begin{equation}
\label{eq:BpfEqn}
\Gamma(s) = G^*\Big( s + \lambda \big(1-\Gamma(s)\big)\Big).
\end{equation}
For relevant standard queueing background see for example \cite{prabhu1998ssp}. 

We now have the following:
\begin{proposition}
\label{prop:statAsymp}
Consider the stationary M/G/1 queue having $g_3 < \infty$. If Assumption~\ref{sweetAss1} holds, then the y-intercept in \eqref{ass1} is given by
\[
\overline{b}_e = 
L_e
\frac{\rho}{(1-\rho)^2}, \,\,\, \mbox{with} \,\,\,
L_e=\frac{
(3 c^4 - 4 \gamma  c^3 + 6 c^2 -1 )\rho^3 +
(4 \gamma c^3  - 12 c^2 + 4)\rho^2 +
(6 c^2 -6)\rho
}{6}.
\]
\end{proposition}
\begin{proof}
Let $\tilde{b}(\cdot)$ be such that $
v(t) = \lambda t + \tilde{b}(t).$ 
By Assumption~\ref{sweetAss1},
$
\lim_{t \to \infty} \tilde{b}(t) = \overline{b}_e.
$
Since the LST of $\lambda t$ is $\lambda/s$ and since the LST is a linear operator, we have from \eqref{eq:DaelyFormula} that for $\Re(s)>0$ the LST of $\tilde{b}(\cdot)$ is
\[
b^*(s)  = \int_0^\infty e^{-st} d \tilde{b}(t).
\]
By standard application of Tauberian theorems (see for example \cite{widder1959laplace}), we have
\begin{equation}
\label{eq:bSlimit}
\lim_{s \to 0} b^*(s) = \overline{b}_e.
\end{equation}
The remainder of the derivation deals with evaluation of  the limit \eqref{eq:bSlimit} by using \eqref{eq:bStarS} together with \eqref{eq:KPformula} in order to find $\overline{b}_e$. This is a combination of straightforward classic queuing calculations together with five applications of L'Hopital's rule. It requires us to evaluate the first three moments of the busy period by taking derivatives of \eqref{eq:BpfEqn} and setting $s\to 0$, and we get
\begin{align*}
b_1 &= (1+\lambda b_1) g_1, \\
b_2 &= \lambda b_2 g_1 + (1+\lambda b_1)^2 g_2, \\
b_3 &= \lambda b_3 g_1 + 3 \lambda (1+ \lambda b_1) b_2 g_2  + (1+\lambda b_1)^3 g_3,
\end{align*}
where, $b_1 = -\Gamma'(0)$, $b_2 = \Gamma''(0)$ and $b_3 = -\Gamma'''(0)$.
These values are well known and appear in many queueing texts, yet we present them here for completeness:
\begin{align*}
b_1 &= \mu^{-1} \frac{1}{1-\rho},\\
b_2 &= \frac{g_2}{(1-\rho)^3} =\mu^{-2}  \frac{c^2+1}{(1-\rho)^3},\\
b_3 &=\frac{g_3 (1-\rho) +3 \lambda g_2^2}{(1-\rho)^5} =
\mu^{-3} \frac{3 c^4 \rho +c^3 \gamma (1 - \rho )+3 c^2 (1+\rho)+2 \rho +1}{(1-\rho)^5}.
\end{align*}

\noindent
\end{proof}

\medskip
\noindent
Here are some observations:
\begin{itemize}
\item If $G$ follows an exponential distribution, then $c=1$ and $\gamma=2$, yielding $\overline{b}_e=0$, as is expected since in this case $\{D(t)\}$ is a Poisson process.
\item $L_e$ (and thus $\overline{b}_e$) is monotone increasing in both $c$ and $\gamma$. 
\item
As $\rho \to 1$, $L_e \to (c^4-1)/2$. This gives some insight into the form of the variance curve of heavy traffic systems, showing how the squared coefficient of variation of $G$ plays a role when $\rho \approx 1$: When $c^2>1$ the y-intercept is positive, and when $c^2<1$ the y-intercept is negative. Further, as typical for heavy traffic systems, only the first two moments of the service time play a role. The skewness coefficient $\gamma$ does not matter when $\rho \approx 1$.
\end{itemize}

\subsection{M/G/1 queue: On a conjecture by Daley}
\label{sec:daley}

In the M/M/1 queue case, (\ref{eq:DaelyFormula}) yields $v^*(s)=\lambda/s$ which corresponds to the variance curve $v(t)=\lambda t$. This is expected since for the stationary M/M/1 queue, $\{D(t)\}$ is a Poisson process.  In \cite{daley1975fso}, Daley conjectures that the reverse direction is also true:
\begin{center}
 {\em Having $v(t)=\lambda t$ implies that the service time distribution is exponential}. 
\end{center}
Restated in terms of LSTs using (\ref{eq:DaelyFormula}), the conjecture is that $b^*(s) = 0$ only for the LST corresponding to a non-negative probability distribution such that $G^*(s) = 1/(1 + s g_1)$ with $g_1 > 0$. Proving Daley's conjecture would generalize a result by Finch, \cite{Finch0281}, stating that all stationary M/G/1 queues with an output Poisson process are M/M/1 queues.

Our expression for $\overline{b}_e$ in Proposition~\ref{prop:statAsymp} gives a necessary condition for $v(t) = \lambda t$:
\begin{center}
 {\em $v(t)=\lambda t$ only if $L_e=0$}. 
\end{center}
\begin{figure}[t]
\begin{center}
\includegraphics[width=12cm]{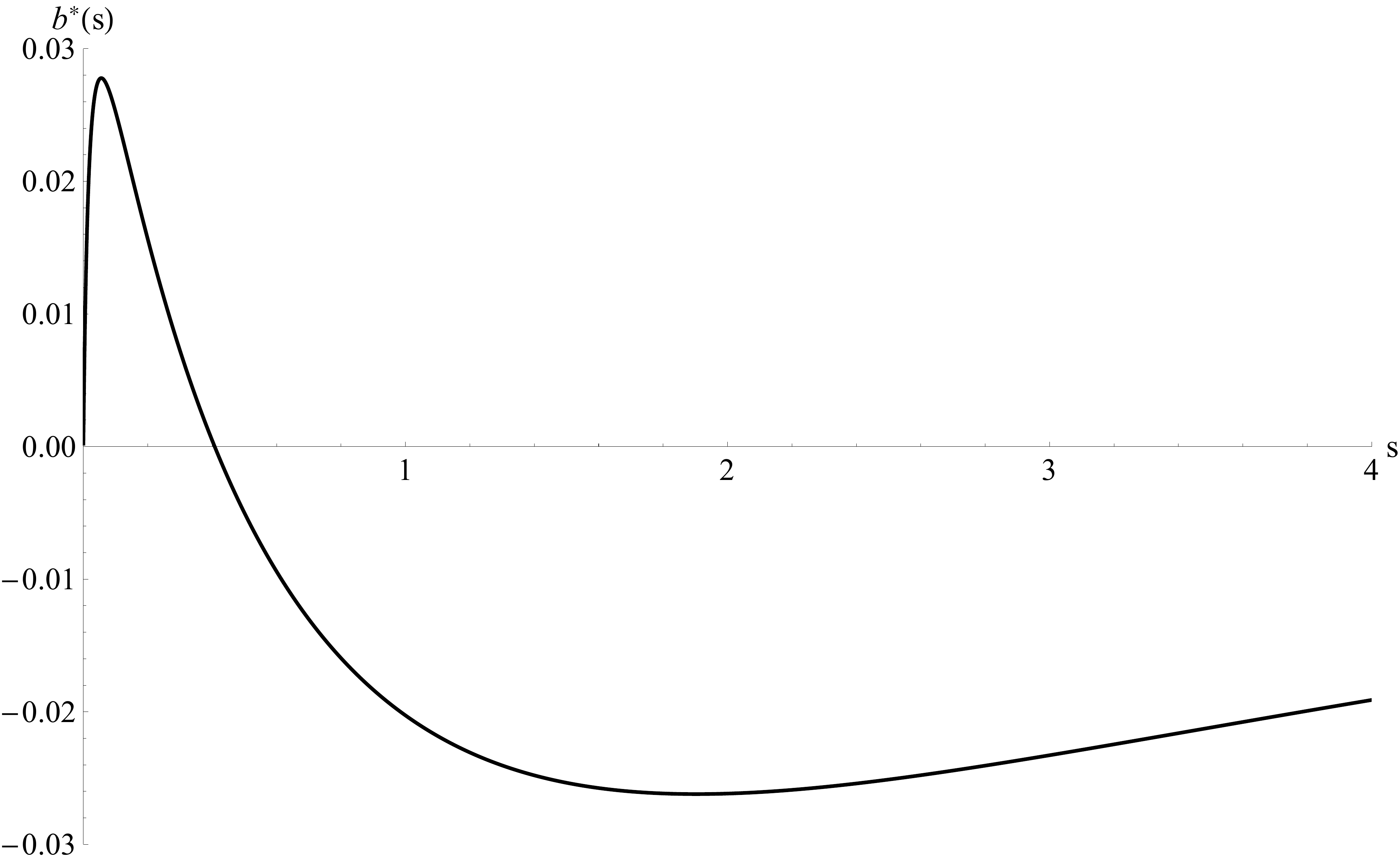}
\end{center}
\caption{A numerical illustration that the claim $L_e=0 \Rightarrow v(t)=\lambda t$ is {\bf not} correct.
\label{figDNon}}
\end{figure}
At this point it is tempting to believe that if $L_e=0$ (i.e. $\overline{b}_e=0$) then $v(t) = \lambda t$.  This could then be used to disprove Daley's conjecture, since $L_e$ only depends on the first three moments of $G$ and it is known that the exponential distribution is not characterized (within the class of non-negative distributions) by the first three moments.  For example, consider a mixture of an exponential random variable with point masses at $1/2,\,\, 3/2,\,\, 5/2$ and $9/2$, with LST: 
\[
G^*(s)=\frac{1}{384}\Big(
192\frac{1}{1+s}
+147e^{-\frac{1}{2}s}
+8e^{-\frac{3}{2}s}
+30e^{-\frac{5}{2}s}
+7e^{-\frac{9}{2}s}
\Big).
\]
As for the mean~$1$ exponential distribution, this distribution yields,
\[
-G^{*'}(0) = 1,
\qquad
G^{*''}(0) = 2,
\qquad
-G^{*'''}(0) = 6.
\]
and thus $L_e=0$.  However, by numerical evaluation of $b^*(s)$, it is clear that $b^*(s) \neq 0$ and thus in this case $v(t) \neq \lambda t$, but rather,
\[
v(t) = \lambda t + o(1),
\]
where the $o(1)$ term is not identically $0$.  The numerical evaluation was performed for the case of $\lambda =3/4$ (and $\mu=1$), and the function $b^*(s)$ is shown in Figure~\ref{figDNon}. Note that in this case, $\Gamma(s)$ was found by iterating \eqref{eq:BpfEqn} for fixed $s$ over a fine grid of $s$. Daley's conjecture thus remains open.

\subsection{M/G/1 queue: Arbitrary initial conditions}

We are now able to use the steady state y-intercept to obtain the y-intercept for a system with an arbitrary distribution of the initial state.

\begin{proposition}
\label{prop:coupling}
Consider the M/G/1 queue with $\rho < 1$, $g_3 < \infty$ and an arbitrary initial distribution of the number of customers having a finite variance. If Assumption~\ref{sweetAss1} holds, then the y-intercept in \eqref{ass1} is given by 
\[
\overline{b}_\theta = \sigma^2_0 - \sigma^2_\pi + \overline{b}_e,
\]
where $\sigma^2_0 = \Var \big( Q(0) \big)$, $\sigma^2_\pi$ is the steady state variance of the number of customers in the system, and $\overline{b}_e$ is as in Proposition~\ref{prop:statAsymp}.
\end{proposition}
\begin{proof}
We use a coupling argument. Consider the following two systems under the same sample path of the arrival process and service times. Denote the systems by $0$ and $\theta$. System~$0$ starts empty and system~$\theta$ starts with $Q(0)$ customers in the system. Operate system~$\theta$ by giving low priority to the initial $Q(0)$ customers, that is, these customers are only served with preemption when there are no other customers in system. This implies that the first $Q(0)$ customers of system $\theta$ are being served only at times that coincide with idle periods of System~$0$. Thus after a finite time $T$, the trajectories of the queue lengths of the two systems coincide. Thus for $t \ge T$, $D_\theta(t) = D_0(t) + Q(0)$, where $D_i(\cdot)$ denote the output counting processes. This yields
\[
\lambda t + \overline{b}_\theta= \lambda t + \overline{b}_0 + o(1) + \Var(Q(0)),
\]
Taking $t \to \infty$ we obtain,
\begin{equation}
\label{eq:couple}
\overline{b}_\theta = \overline{b}_0 + \Var(Q(0)).
\end{equation}
Selecting now $Q(0)$ to be distributed according to the steady state distribution we get from \eqref{eq:couple},
\[
\overline{b}_0 = \overline{b}_e - \sigma^2_\pi.
\]
Applying this again in (\ref{eq:couple}) we obtain the result. Note that a similar coupling argument also holds for the non-preemptive case.
\end{proof}
\medskip

As is well known, $\sigma^2_\pi$ can be obtained directly from $\Pi(\cdot)$, yet it is a cumbersome calculation. We state it here for completeness:
\begin{equation}
\label{eq:StatPi}
\sigma^2_\pi =
\Big(\big( \frac{1}{4} c^4 -\frac{1}{3} \gamma c^3 +\frac{1}{2} c^2 -\frac{1}{12} \big) \rho^3
+ \big(\frac{1}{3} \gamma  c^3 - \frac{3}{2} c^2 + \frac{5}{6} \big) \rho^2 +
\big(  \frac{3}{2} c^2 - \frac{3}{2} \big) \rho + 1
\Big)
\frac{
 \rho
}
{(1-\rho)^2}.
\end{equation}
As a result of the above, the y-intercept for a system that starts empty is:
\[
\overline{b}_0 = 
-(1-L_0)
\frac{\rho}{(1-\rho)^2}, \,\,\, \mbox{with} \,\,\,
L_0=\frac{
(3 c^4 - 4 \gamma c^3  + 6 c^2 - 1)\rho^3+
(4 \gamma c^3 - 6 c^2 -  2)\rho^2+
(-6 c^2 + 6)\rho
}{12}.
\]
\medskip
\noindent
Here are some observations:
\begin{itemize}
\item In the M/M/1 queue, $L_0=0$ and thus $\overline{b}_0=-\rho/(1-\rho)^2$. As expected, this is in agreement with the case $\rho<1$ in~\eqref{eq:mm1Var}. 
\item As $\rho \to 1$, $L_0 \to (c^2-1)^2/4$. This implies that in heavy-traffic, $c^2$ plays a similar role with respect to the sign of the y-intercept as it did in the stationary case: In this case when $c^2>3$ the y-intercept is positive, and when $c^2 \le 3$ the y-intercept is negative. Compare with the remarks following Proposition~\ref{prop:statAsymp}.
\end{itemize}

\subsection{M/G/1 queue: Asymptotic covariance}

We now exploit the y-intercept of the stationary system, $\overline{b}_e$, to find the asymptotic covariance between $A(t)$ and $Q(t)$, that is,
\[
c_{A,Q} := \lim_{t \to \infty} \Cov \big( A(t), Q(t) \big).
\]
In similar spirit to the y-intercept term, we need some further assumptions to ensure the above limit exists:
\begin{ass}
\label{sweetAss2}
\[
\lim_{t \to \infty} \Var\big(Q(t)\big) = \sigma^2_\pi.
\]
\end{ass}
\begin{ass}
\label{sweetAss3}
\[
\lim_{t \to \infty} \Cov\big(Q(0),Q(t)\big) = 0.
\]
\end{ass}
As conveyed by personal communication with Fralix, Assumptions \ref{sweetAss2} and \ref{sweetAss3} hold for the stable FCFS M/G/1 queue having a finite third moment. Nevertheless, we are not able to establish that they hold for the M/G/1 queue with an arbitrary work-conserving scheduling policy.  Note that Assumption~\ref{sweetAss2} would follow from uniform integrability of the sequence $\{Q(t)^2,\, t \ge 0\}$. It is also worth mentioning that Assumption~\ref{sweetAss2} does not hold in general for positive recurrent Markov chains with finite stationary variance. We conjecture that both of the above assumptions hold for stable ($\rho <1$) M/G/1 queues operating under a work-conserving policy, with $\Var\big(Q(0)\big)<\infty$ and $g_3<\infty$.

This is our result regarding $c_{A,Q}$:
\begin{proposition}
\label{thm:covConstant}
Consider the M/G/1 queue with $\rho < 1$ and  $g_3 < \infty$  and arbitrary distribution of $Q(0)$ having a finite variance. If Assumptions~\ref{sweetAss1}, \ref{sweetAss2} and \ref{sweetAss3} hold then,
\[
\lim_{t\to\infty} \Cov \big(A(t),Q(t)\big) =
\frac{\rho}{(1-\rho)^2}\Big(
1+
(c^2-1)\frac{\rho(2-\rho)}{2}
\Big).
\]
\end{proposition}
\begin{proof}
Taking the variance of $D(t)=A(t)+Q(0)-Q(t)$ and using Proposition~\ref{prop:coupling} (together with $g_3<\infty$ and Assumption~\ref{sweetAss1}) for the left hand side, we obtain
\begin{equation}
\nonumber
\lambda t + \Var(Q(0)) - \sigma^2_\pi + \overline{b}_e  + o(1)= \lambda t + \Var(Q(0)) + \Var(Q(t)) - 2 \Cov(Q(0),Q(t)) - 2 \Cov (A(t),Q(t)).
\end{equation}
Now canceling terms, and using Assumptions~\ref{sweetAss2} and \ref{sweetAss3}, we obtain,
\[
\lim_{t\to\infty} \Cov\big(A(t),Q(t)\big) = \sigma^2_\pi - \frac{\overline{b}_e}{2}.
\]
Using the expressions of $\overline{b}_e$ and $\sigma^2_\pi$, we obtain the result after some simplification.
\end{proof}

\medskip

Note that the above result implies that for the M/M/1 queue,
\[
\lim_{t \to \infty} \Cov\big(A(t),Q(t)\big) = \frac{\rho}{(1-\rho)^2}.
\]
We are not aware of an alternative derivation of this quantity for the M/M/1 queue.

\section{Conclusion}
\label{sec:conc} 

In going through the detailed MAP derivations for M/M/1/K queues, we have illustrated how asymptotic quantities such as $\overline{b}$ may be obtained explicitly. The key is to have explicit expressions for mean hitting times in the underlying Markov chain. Further, by using the Drazin inverse we have gained some further insight into the BRAVO effect. In plotting the graphs of $\overline{d}_v$ and $\overline{d}_b$, we observe that spikes occur when $\lambda \approx \mu$ in a similar fashion to the BRAVO effect. 

For the M/G/1 queue, the formal calculations are of a different flavor, but are also generally tedious. Nevertheless, in stating our results, we have had to resort to Assumptions~\ref{sweetAss1}--\ref{sweetAss3}, which, as far as we know, are not established.  We believe that these assumptions hold when $G$ has a finite third moment.  Personal communication with Brian Fralix has demonstrated the validity of Assumptions~\ref{sweetAss2} and \ref{sweetAss3} in the FCFS case, yet his methods rely on this restriction.  We believe these hold in greater generality and leave this challenge for future work.  Towards this end, it is worthwhile to refer the reader to \cite{fralix2012time} and \cite{fralix2010new}, where similar transient analysis of the M/G/1 queue is undertaken; see also \cite{pakes71} for classic results in the stationary case.

Besides the ``transient moment problems'' associated with Assumptions~\ref{sweetAss1}--\ref{sweetAss3}, our work has highlighted two other open questions. The first is Daley's conjecture discussed in detail in Section~\ref{sec:daley} above. It remains open.  The second is finding formulas for $\overline{b}_\theta$ when $\rho>1$.  In the case $\rho=1$ we conjecture this term does not exist (as for the M/M/1 queue), but for the case $\rho>1$ it is an interesting challenge to search for a closed form formula in terms of the moments of $G$.

\vspace{30pt}
\begin{center}{\bf Acknowledgment}\end{center}
We thank Onno Boxma, Brian Fralix and Guy Latouche for useful discussions and advice. Yoni Nazarathy is supported by Australian Research Council (ARC) grants DP130100156 and DE130100291. Yoav Kerner is supported by Israeli Science Foundation (ISF) grant 1319/11.   Yoav Kerner and Yoni Nazarathy also thank EURANDOM for hosting and support. Sophie Hautphenne and Peter Taylor are supported by Australian Research Council (ARC) grant DP110101663.

\bibliography{LitDB}

\appendix
\end{document}